\documentclass[a4paper,oneside,11pt]{amsart}

\reversemarginpar

\usepackage[colorlinks,hyperindex,linkcolor=blue,urlcolor=black,pdftitle={On power bounded operators that are quasiaffine transforms of singular unitaries},pdfauthor={Maria F.Gamal'}]
{hyperref}

\usepackage{amsmath}
\usepackage{amsfonts}
\usepackage{amssymb}
\usepackage{amsthm}

\usepackage[english]{babel}
\usepackage{enumerate}

\usepackage{graphicx}
\usepackage{color}

\usepackage[abbrev]{amsrefs}

\numberwithin{equation}{section}

\theoremstyle{plain}
\newtheorem{theorem}{Theorem}[section]
\newtheorem{lemma}[theorem]{Lemma}
\newtheorem{corollary}[theorem]{Corollary}

\theoremstyle{definition}
\newtheorem{remark}[theorem]{Remark}
\newtheorem{example}[theorem]{Example}
\newtheorem{definition}[theorem]{Definition}

\begin{document}

\title[On power bounded operators]{On power bounded operators that are quasiaffine transforms of singular unitaries}
\author{Maria F. Gamal'}
\address{
 St. Petersburg Branch\\ V. A. Steklov Institute 
of Mathematics\\
 Russian Academy of Sciences\\ Fontanka 27, St. Petersburg\\ 
191023, Russia  
}
\email{gamal@pdmi.ras.ru}





\begin{abstract}In \cite{9} a question is raised:
  if a power bounded operator 
is quasisimilar to a singular unitary operator, is it similar
to this unitary operator? For polynomially bounded operators,
a positive answer to this question is known \cite{1}, \cite{13}. In this paper
a positive answer is given in some particular cases, 
but in general an answer remains unknown.
\end{abstract} 

\maketitle

\section{Introduction}


Let $\mathcal H$ be a (complex, separable) Hilbert space, and let $T$ be 
a (linear, bounded) operator acting on 
$\mathcal H$. An operator $T$ is called {\it power bounded} if 
$\sup_{n\geq 0}\|T^n\|<\infty$.

Let $T$ and $R$ be operators on spaces $\mathcal H$ and $\mathcal K$, 
respectively, and let
$X:\mathcal H\to\mathcal K$ be a linear bounded transformation such that 
$X$ {\it intertwines} $T$ and $R$,
that is, $XT=RX$. If $X$ is unitary, then $T$ and $R$
are called {\it unitarily equivalent}, in notation:
$T\cong R$. If $X$ is invertible, that is, the inverse $X^{-1}$ is bounded,  then $T$ and $R$ 
are called {\it similar}, in notation: $T\approx R$.
If $X$ is a {\it quasiaffinity}, that is, $\ker X=\{0\}$ and 
$\operatorname{clos}X\mathcal H=\mathcal K$, then
$T$ is called a {\it quasiaffine transform} of $R$, 
in notation: $T\prec R$. If $T\prec R$ and 
$R\prec T$, then $T$ and $R$ are called {\it quasisimilar}, 
in notation: $T\sim R$. Recall that if $T$ and $R$ 
are unitary operators and  $T\prec R$, then  $T\cong R$
({\cite[II.3.4]{19}}). 

 It is well known that if $T$ is a contraction, that is, 
 $\|T\|\leq 1$, and $U$ is a singular unitary operator, then the relation  
$T\prec U$  implies that $T\cong U$.

It is known that if $T$ is a power bounded operator 
and $U$ is a unitary operator whose spectral measure 
is pure atomic, then the relation $T\prec U$ 
implies that $T\approx U$ (\cite{1}, \cite{17}, see also Example \ref{exa2.6}).

Also, if $T$ is polynomially bounded, that is,  
there exists a constant $M$ such that 
$\|p(T)\|\leq M \sup\{|p(z)|, \ |z|\leq 1\}$
for every polynomial $p$, 
and $U$ is singular unitary, then the relation 
$T\prec U$ implies that $T\approx U$ \cite{1}, \cite{13}. 
We sketch the proof for the case, where the 
 (closed) spectrum of $U$ has Lebesgue measure zero. 
 Denote by $\mathbb D$ and by $\mathbb T$ the unit disk and the unit circle,
 respectively.
 For a polynomially bounded operator $T$ 
the functional calculus on the disk algebra $\mathcal A(\mathbb D)$ is well defined. 
Let $E\subset\mathbb T$ be a closed set of zero  Lebesgue measure. 
We denote by $C(E)$ the space of continuous functions on $E$.
  We set $\mathcal I(E)=\{f\in\mathcal A(\mathbb D):\  f=0$ on $E\}$ and 
$\mathcal A(E) =\mathcal A(\mathbb D)/\mathcal I(E)$. 
It is known that  the natural imbedding $\mathcal A(E) \to C(E)$
is an isometrical isomorphism {\cite[ch.6]{3}}. 
Applying this fact to a polynomially bounded operator $T$ such that
$T\prec U$, where $U$ is unitary with spectrum $E$, 
we obtain that $\|T^{-n}\|\leq M$
for $n\in\mathbb N$ (where $M$ is a constant from the condition 
on polynomially boundedness of $T$).
 Thus, we can apply \cite{18} and conclude that $T\approx U$.

If we suppose only that $T$ is power bounded, then it seems 
 natural to consider
the functional calculus on the subalgebra 
$$A^+(\mathbb T)=\{f(z)=\sum_{n\geq 0}\hat f(n) z^n, \ \ 
\|f\|=\sum_{n\geq 0}|\hat f(n)|<\infty\}$$
of the  Wiener algebra 
$$A(\mathbb T)=\{f(z)=\sum_{n\in \mathbb Z}\hat f(n) z^n, \ \ 
\|f\|=\sum_{n\in \mathbb Z}|\hat f(n)|<\infty\}.$$
 Let $E\subset\mathbb T$ be a closed set of zero  Lebesgue measure. 
As before, we set  
$$I^+(E)=\{f\in A^+(\mathbb T):\  f=0\ \text{on}\ E\}, 
\ \ \ A^+(E) =A^+(\mathbb T)/I^+(E),$$
and we consider the natural imbedding $ A^+(E) \to C(E)$. 
But there exist many  closed sets $E$
of zero  Lebesgue measure for 
which this imbedding is not an  isomorphism \cite{2}, \cite{6}, \cite{8}, \cite{11}.
Let us consider $AA^+$-sets \cite{6}, \cite{7}. Namely, 
for a closed set $E\subset\mathbb T$ of zero Lebesgue measure, we set 
$$I(E)=\{f\in A(\mathbb T):\  f=0 \ \text{on}\  E\}, \ \ \ A(E) =A(\mathbb T)/I(E),$$
and consider the natural imbedding 
\begin{equation}\label{1.1} A^+(E) \to A(E). \end{equation}
If the imbedding \eqref{1.1} is onto, then the set $E$ is called an $AA^+$-set.
For an $AA^+$-set $E$ the imbedding \eqref{1.1}  is invertible; we will denote by 
 $K(E)$ the norm of its inverse.

\begin{lemma}\label{lem1.1} Suppose that $T$ is a power bounded operator, 
$U$ is a unitary operator, 
$T\prec U$ and
the spectrum $E$ of $U$ is an $AA^+$-set. Then $T\approx U$ and 
\begin{equation}\label{1.2}\sup_{n\leq 0}\|T^n\|\leq K(E)\sup_{n\geq 0}\|T^n\|, \end{equation}
where $K(E)$ is the norm of the inverse to the imbedding \eqref{1.1}.\end{lemma}

\begin{proof} We set $M=\sup_{n\geq 0}\|T^n\|$.
If  $f\in A^+(\mathbb T)$, then $f(T)$ is well defined, 
and $\|f(T)\|\leq M\|f\|_{A^+(\mathbb T)}$.
Since $T\prec U$, we have $f(T)\prec f(U)$. If $f\in I^+(E)$, then $f(U)=0$, 
and from the relation $f(T)\prec f(U)$ we conclude that 
$f(T)=0$. Thus, the functional calculus for $T$ on the algebra $A^+(E)$
is well defined. Let $k\in \mathbb N$ be fixed. Since $E$ is an $AA^+$-set, there exists 
a function $f_k\in A^+(\mathbb T)$
such that $f_k(\zeta)= \zeta^{-k}$ for every $\zeta\in E$. 
Clearly, $f_k(U)=U^{-k}$.
We have $ f_k(T)T^k = T^k  f_k(T)\prec U^k  f_k(U)=I$, 
and we conclude that  $ f_k(T)T^k = T^k  f_k(T)=I$.
Thus, $T^k$ is invertible, and $T^{-k}= f_k(T)$. Furthermore, 
$$\|T^{-k}\|=\| f_k(T)\|\leq M\|f_k\|_{A^+(E)}\leq MK(E) 
\|\zeta^{-k}\|_{A(E)}\leq MK(E).$$
\end{proof}

Again,  there exist many  closed sets of zero  Lebesgue measure 
that are not $AA^+$-sets.
Moreover, for an $AA^+$-set $E$, the norm $K(E)$ of the inverse 
to  imbedding  \eqref{1.1}  can be arbitrary large \cite{5}, \cite{7}.   
A question arises if estimate   \eqref{1.2}  exact.  
It will be shown in this paper that, for sets  satisfying 
some metric condition (see Definition \ref{def3.1}),
this estimate is not exact. Namely,  for every $K>0$ 
there exists an $AA^+$-set $E$ such that 
$K(E)\geq K$  \cite{5}, \cite{7}, but the  estimate of the left part of 
 \eqref{1.2} depends only on $\sup_{n\geq 0}\|T^n\|$
(Theorem \ref{th3.5}).
 
The paper is organized as follows. In Section 2, some general propositions
on power bounded operators that are quasisimilar 
to singular unitaries are proved. In Section 3, these propositions 
are applied to unitary operators whose spectral measure 
is supported on the sets satisfying some metric condition.

In the rest of Introduction, the notation and definitions are given.

 By $I$ the identity operator is denoted; if it is needed, 
the space on which it acts will be mentioned in  the lower index.
Let $T$ be a power bounded operator on a Hilbert space $\mathcal H$.
$T$ is of class $C_{1\cdot}$, if $\inf_{n\geq 0}\|T^nx\| > 0$
for every $x\in \mathcal H$, $x\neq 0$. $T$ is of class
 $C_{0\cdot}$, if  $\lim_n\|T^nx\| = 0$
for every $x\in \mathcal H$.
By $C_{\cdot 1}$ and   $C_{\cdot 0}$ the classes of
 power bounded operators $T$
such that  $T^\ast\in C_{1\cdot}$ and   $T^\ast\in C_{0\cdot}$, 
respectively, are denoted. 
As usually, $C_{11}= C_{1\cdot}\cap C_{\cdot 1}$,
  $C_{10}= C_{1\cdot}\cap C_{\cdot 0}$, and so on.
By $\operatorname{Lat}T$,  $\operatorname{Hyplat}T$, and 
 $\operatorname{Hyplat}_1T$ the invariant subspace lattice, 
the hyperinvariant subspace lattice, and the lattice 
of hyperinvariant subspaces
for power bounded operator $T$  such that the  restrictions of $T$ 
 on these subspaces belong to the class $C_{11}$, respectively,
are denoted (see \cite{9}).

Let $U$ be a unitary operator on a Hilbert space $\mathcal H$. There exists a finite positive Borel measure $\mu$ on the unit circle $\mathbb T$ having the following property:
for every $x$, $y\in \mathcal H$ there exists a function $f_{x,y}\in L^1(\mu)$ 
such that $(U^n x,y)=\int_{\mathbb T}\zeta^n  f_{x,y}(\zeta)d\mu(\zeta)$ for every
$n\in\mathbb Z$
(of course, $\int_\sigma f_{x,y}d\mu = (\mathbb E(\sigma) x,y)$, 
where  $\mathbb E$ is the (operator valued) spectral
measure of $U$, and $\sigma\subset\mathbb T$ is a Borel set). 
The measure $\mu$ is called the scalar spectral measure of $U$.
In other words, the (operator valued) spectral measure and the scalar 
spectral measure of a unitary operator are mutually absolutely continuous.
A unitary operator is called singular, 
if its spectral measure is singular with respect to Lebesgue measure 
(arc length  measure) on  $\mathbb T$. Recall that every  invariant subspace
of a singular
unitary operator is reducing, that is, 
its orthogonal complement is also invariant, therefore, the restriction of
 a singular unitary operator on its  invariant subspace is a unitary operator.    
Let  a sequence $\{n_k\}_k$ and a function $\varphi\in L^\infty(\mu)$ be such that
$\zeta^{n_k}\to_k \varphi$ in the weak-star topology on $L^\infty(\mu)$.
Then $U^{n_k}\to_k  \varphi(U)$ in the weak operator topology.

Let $T$ be an operator on a Hilbert space. Then $T$ is similar to a unitary operator
 if and only if $\sup_{n\in\mathbb Z}\|T^n\|<\infty$ (\cite{18}, see also \cite{19}).
We will  use this fact in what follows without additional references.

To conclude Introduction, we note that in \cite{14} an example 
of a power bounded operator which is quasisimilar to a unitary 
operator and is not similar to a contraction is constructed, 
but the unitary operator from this example is
 the bilateral shift of infinite multiplicity.

\section{Some general results}

In this section, we prove the main result of the paper (Theorem \ref{th2.11}), 
which claims that countable orthogonal sum of ``good" 
(in the sense of the present paper) unitary operators rests ``good". 

 The following simple lemmas are given for convenience of references.

\begin{lemma}\label{lem2.1} Let $T:\mathcal H\to\mathcal H$ be a power bounded operator. 
We set $M=\sup_{n\geq 0}\|T^n\|$.
Then  $$\limsup_n \|T^nx\|\leq M\liminf_n \|T^nx\| \ \ \text{for every} 
\ \  x\in \mathcal H.$$\end{lemma}

\begin{proof} Let $x\in \mathcal H$ be fixed, and let the sequences $\{n_k\}_k$,  
$\{\ell_j\}_j$ be such that  
$\limsup_n \|T^nx\|=\lim_j \|T^{\ell_j}x\|$, 
$\liminf_n \|T^nx\|=\lim_k \|T^{n_k}x\|$.
There exists a sequence $\{j_k\}_k$ such that 
$\ell_{j_k}>n_k$, $k=1,2,\ldots$. We have
$$\limsup_n \|T^nx\|=\lim_k \|T^{\ell_{j_k}}x\|=
\lim_k \|T^{\ell_{j_k}-n_k}T^{n_k}x\|
$$
$$\leq M\lim_k \|T^{n_k}x\|= M\liminf_n \|T^nx\|.$$ \end{proof} 

\begin{lemma}\label{lem2.2} $1)$ Let $U$ be a unitary operator having 
the following property:

\noindent $(\ast)$ if $T$ is a power bounded operator 
such that $T\prec U$, then $T\approx U$.

\noindent Let $\mathcal E\in \operatorname{Lat}U$. 
Then $U|_{\mathcal E}$ has  property $(\ast)$.

 $2)$ Let $U$ be a unitary operator having the following property:

\noindent $(\ast\ast)$ if $T$ is a power bounded operator such that $T\sim U$, 
then $T\approx U$.

\noindent Let $\mathcal E\in \operatorname{Lat}U$. 
Then $U|_{\mathcal E}$ has  property $(\ast\ast)$.\end{lemma}

\begin{proof}  
Let $U$ be a unitary operator having  property  $(\ast)$. 
Then $U$ is singular.
Indeed, if  $U$ has 
an absolutely continuous part, then one can find a contraction $T$ such that
 $T\sim U$, but $T$ is not similar to $U$ {\cite[II.3.5, VI.3.5, IX.1.2]{19}}. Since $U$ is singular, every invariant subspace  $\mathcal E$ of $U$ is reducing, that is, $\mathcal E^\perp$ is also invariant.

Let $U$ has property $(\ast)$,  let   $\mathcal E\in \operatorname{Lat}U$, and 
let $T_1$ be a power bounded operator such that 
$T_1\prec U|_{\mathcal E}$. We set
$T=T_1\oplus U|_{\mathcal E^\perp}$. Clearly, 
$T\prec U|_{\mathcal E}\oplus U|_{\mathcal E^\perp}\cong U$, therefore,  $T\approx U$.
Since $T_1$ is the  restriction of $T$ on its invariant subspace, 
 $T_1$ is similar to the  
restriction of $U$ on some of  its invariant subspace,  
thus, $T_1$ is similar to some unitary operator.
Since $T_1\prec U|_{\mathcal E}$,  we have   $T_1\approx U|_{\mathcal E}$ {\cite[II.3.4]{19}}, see Introduction of the paper.
Part 1 of the lemma is proved. The proof of part 2
is analogous.\end{proof}

\begin{lemma}\label{lem2.3} Suppose that $R$, $A:\mathcal K\to\mathcal K$ are operators and  
  $\{n_k\}_k$ is a sequence such that $R^{n_k}\to_k A$ in 
the  weak operator topology.
Suppose that $T:\mathcal H\to\mathcal H$ is a power bounded operator, and 
$X:\mathcal H\to\mathcal K$ is a quasiaffinity such that $XT=RX$. Then there exists
 an operator $B:\mathcal H\to\mathcal H$ such that  $T^{n_k}\to_k B$ in 
 the weak operator topology,
$XB=AX$, and $\|B\|\leq\sup_{n\geq 0}\|T^n\|$.\end{lemma}

 \begin{proof} We set $M=\sup_{n\geq 0}\|T^n\|$. Let $x\in\mathcal H$ be fixed. 
If $g\in\mathcal K$, then 
\begin{equation}\label{2.1}(T^{n_k}x,X^\ast g)=(XT^{n_k}x, g)=(R^{n_k}Xx, g)\to_k(AXx,g). \end{equation}
Since
\begin{equation}\label{2.2}\sup_{k\geq 0}\|T^{n_k}x\|\leq M\|x\|<\infty \end{equation} 
and $X^\ast \mathcal K$ is dense in $\mathcal H$, we have that
 $\lim_k(T^{n_k}x,y)$ exists for every $y\in\mathcal H$.
By the Banach--Steinhaus theorem, there exists $h\in\mathcal H$ such that 
$$\lim_k(T^{n_k}x,y)=(h,y) \ \ \text{  for every } \ y\in\mathcal H.$$ We set $Bx=h$. From \eqref{2.2} we conclude that 
 $\|Bx\|\leq M\|x\|$.
It is clear from the definition of $B$ and from \eqref{2.1} that  $T^{n_k}\to_k B$ 
in the weak operator topology and $XB=AX$.  \end{proof} 

The following lemma shows that, for some unitary operators, the assumption 
on quasisimilarity in the question regarded in the paper can be replaced by 
the assumption that a power bounded operator is a quasiaffine transform of 
this unitary. We note that unitary operators satisfying the conditions of 
Theorems \ref{th2.5}, \ref{th2.12}, and \ref{th3.5} satisfy the conditions of Lemma \ref{lem2.4}, too.

\begin{lemma}\label{lem2.4} Let $U:\mathcal K\to\mathcal K$ be a unitary operator 
having the following properties:

$1)$ if  $x\in\mathcal K$ is such that  $U^n x\to_n 0$ in the weak topology,
then $x=0$;

$2)$  if $T$ is a power bounded operator such that $T\sim U$, then $T\approx U$.

Then  if $T$ is a power bounded operator such that $T\prec U$, 
then $T\approx U$.\end{lemma}

 \begin{proof}  Let $T:\mathcal H\to\mathcal H$ be a power bounded operator.
By {\cite[Lemma 1]{9}}, there exists $\mathcal H_1\in\operatorname{Lat}T$ having the following properties:
 $T|_{\mathcal H_1}\in C_{\cdot 1}$, 
$P_{\mathcal H\ominus \mathcal H_1} T | _{\mathcal H\ominus \mathcal H_1} \in C_{\cdot 0}$,
and if $\mathcal E\in\operatorname{Lat}T$ is such that   $T|_{\mathcal E}\in C_{\cdot 1}$,
then $\mathcal E\subset\mathcal H_1$  (by $P_{\mathcal H\ominus \mathcal H_1}$ 
the orthogonal projection on the space $\mathcal H\ominus \mathcal H_1$ is denoted;
 the definitions of classes 
$C_{\cdot 1}$ and $C_{\cdot 0}$ are recalled in Introduction). 
The space $\mathcal H_1$ is called the $C_{\cdot 1}$-subspace of $\mathcal H$ with respect to $T$.

Now let $X:\mathcal H\to\mathcal K$ be a quasiaffinity such that $XT=UX$. 
We set $T_1=T|_{\mathcal H_1}$ and  $\mathcal K_1=\operatorname{clos}X\mathcal H_1$.
Clearly,  $U|_{\mathcal K_1}$ is unitary and $T_1\prec U|_{\mathcal K_1}$.
By the  definition of the space ${\mathcal H_1}$, we have $T_1\in C_{11}$, therefore,
$T_1$ is quasisimilar to a unitary operator \cite{9}.
 From the relation  $T_1\prec U|_{\mathcal K_1}$ and {\cite[II.3.4]{19}}
(see Introduction) we conclude that 
   $T_1\sim U|_{\mathcal K_1}$.
By assumption 2) and Lemma \ref{lem2.2}, $T_1\approx U|_{\mathcal K_1}$.

Since $T_1=T|_{\mathcal H_1}$ is similar to a unitary operator, 
 there exists $\mathcal H_0\in\operatorname{Lat}T$ such that 
$\mathcal H=\mathcal H_1\dotplus\mathcal H_0$
(see \cite{10}; although the proposition in  \cite{10} is formulated for contractions, 
application of  results from  \cite{9} allows to repeat the proof for 
power bounded operators). 
We set $T_0=T|_{\mathcal H_0}$. We will show that    $T_0\in C_{\cdot 0}$.
Let  $\mathcal H_{01}$ be the $C_{\cdot 1}$-subspace of $\mathcal H_0$ 
with respect to $T_0$. We set
 $\mathcal H_{11}=\mathcal H_1\dotplus\mathcal H_{01}$. We have  
 $\mathcal H_{11}\in\operatorname{Lat}T$
and $T|_{\mathcal H_{11}}\in C_{11}$. 
Since $\mathcal H_1$ is  the $C_{\cdot 1}$-subspace of $\mathcal H$ 
with respect to $T$,
we have  $\mathcal H_{11}\subset\mathcal H_1$. Thus, we conclude that
 $\mathcal H_{01}=\{0\}$, that is, $T_0\in C_{\cdot 0}$.

We set $\mathcal K_0=\operatorname{clos}X\mathcal H_0$, and 
 we apply {\cite[Lemma 3.2]{14}} to the operators  $T_0$ and $U|_{\mathcal K_0}$.
Namely, from the relations $T_0\prec U|_{\mathcal K_0}$ and $T_0\in C_{\cdot 0}$
we obtain that $(U|_{\mathcal K_0})^{\ast n}\to_n 0$ in the weak operator topology.
But it means that $U^nx\to_n 0$ in the weak topology for every $x\in\mathcal K_0$.
 By assumption 1) of the lemma, $\mathcal K_0=\{0\}$, therefore,  $\mathcal H_0=\{0\}$.

Thus,  $\mathcal H_1 =\mathcal H$,  $\mathcal K_1 =\mathcal K$, and $T_1=T\approx U$.
 \end{proof} 

The following theorem is a modification of  {\cite[Theorem 2]{1}}.

\begin{theorem}\label{th2.5} Suppose that $U$ is a unitary operator, 
$A$ is an invertible operator,  
and $\{n_k\}_k$ is a sequence such that $U^{n_k}\to_k A$ 
in the weak operator topology.
Let $T$ be a power bounded operator such that $T\sim U$. We set  
$M=\sup_{n\geq 0}\|T^n\|$. Then $T\approx U$ and 
 $\sup_{n\leq 0}\|T^n\|\leq\|A^{-1}\|M^2$.\end{theorem}

 \begin{proof}  Since  $T\sim U$, we have that $U$ is unitarily equivalent 
to the isometric asymptote of $T$.
Denote by $\mathcal H$ and $\mathcal K$ the spaces on which $T$ and $U$ act, and by
 $X:\mathcal H\to\mathcal K$ the canonical intertwining mapping. It is easy to see 
 from the definition of $X$ and 
properties of a Banach limit that
\begin{equation}\label{2.3} \liminf_n \|T^nx\|\leq \|Xx\|\leq\limsup_n \|T^nx\| \ \ \text{for every} 
 \ \  x\in \mathcal H \end{equation}
(see \cite{9}).  Clearly, $\|X\|\leq M$.
Of course, $U$, $A$, $T$  and $X$ satisfy the conditions of Lemma \ref{lem2.3}, 
therefore, there exists an
operator   $B:\mathcal H\to\mathcal H$ such that  $T^{n_k}\to_k B$ 
in the weak operator topology,
$XB=AX$, and $\|B\|\leq M$. Since $A$ is invertible, we have $X=A^{-1}XB$, 
therefore, 
\begin{equation}\label{2.4} \|Xx\|\leq\|A^{-1}\|M\|Bx\|\ \ \text{for every} \ \  x\in \mathcal H. \end{equation}
Clearly, $TB=BT$.

Let $\ell\in\mathbb N$ be fixed. Let $x\in\mathcal H$. Since  $T^{n_k}x\to_k Bx$  
weakly, 
$$\|Bx\|\leq\liminf_k\|T^{n_k}x\|=\liminf_k\|T^{n_k-\ell}T^\ell x\|
\leq\limsup_n\|T^nT^\ell x\|$$
$$\leq M\liminf_n\|T^nT^\ell x\|
\leq M\|XT^\ell x\|\leq\|A^{-1}\|M^2\|BT^\ell x\|=
\|A^{-1}\|M^2\|T^\ell Bx\|  $$
(we apply Lemma \ref{lem2.1}, \eqref{2.3}, and \eqref{2.4}). We obtain that 
\begin{equation}\label{2.5} \|Bx\|\leq\|A^{-1}\|M^2\|T^\ell Bx\|
\ \ \text{for every} \   x\in \mathcal H 
\ \text{and for  every} \  \ell\in\mathbb N.\end{equation}

Since $T\sim U$, there exists a quasiaffinity $Y:\mathcal K\to\mathcal H$ such that
$YU=TY$. It is easy to see from this intertwining relation and 
the conditions 
$U^{n_k}\to_k A$ and $T^{n_k}\to_k B$ in the weak operator topology that $YA=BY$.
We have $\mathcal H=\operatorname{clos}YA\mathcal K=\operatorname{clos}BY\mathcal K$, 
thus, 
$\operatorname{clos}B\mathcal H=\mathcal H$; from this equality and  \eqref{2.5} 
we conclude that
\begin{equation}\label{2.6} \|x\|\leq\|A^{-1}\|M^2\|T^\ell x\|
\ \ \text{for every} \   x\in \mathcal H \ \text{and for every} 
\  \ell\in\mathbb N.\end{equation}
From  \eqref{2.6} we conclude that 
 $T$ is left invertible. The  left invertibility of $T$ and 
  the relation $T\sim U$
imply  that  $T$ is invertible. From \eqref{2.6} we have that 
  $\|T^{-\ell}\|\leq \|A^{-1}\|M^2$
  for every $\ell\in\mathbb N$. By \cite{18}, $T$ is similar to a unitary. Since
     $T\sim U$, this unitary is unitarily equivalent to $U$.
 \end{proof}

Now we mention some known examples of unitary operators that satisfy 
the conditions of Theorem \ref{th2.5}.

\begin{example}\label{exa2.6}  \cite{1}, \cite{17}. Let the spectral measure of 
a unitary operator $U$ be pure atomic, 
that is, there exist countable families of Hilbert spaces 
$\{\mathcal H_j\}_j$ and of points $\{\lambda_j\}_j\subset\mathbb T$ such that 
$U\cong\oplus_j\lambda_jI_{\mathcal H_j}$.
Using the well-known diagonal process, one can find 
a sequence $\{n_k\}_k$ and a family of points
$\{\xi_j\}_j\subset\mathbb T$ such that 
$\lambda_j^{n_k}\to_k\xi_j$ for every $j$.
We set $A=\oplus_j\xi_jI_{\mathcal H_j}$. Then $U^{n_k}x\to_k Ax$ for every $x\in \oplus_j\mathcal H_j$. 
Consequently, $U^{n_k}\to_k A$ 
in the weak operator topology. Clearly,
$A$ is unitary, therefore, $\|A^{-1}\|=1$.
\end{example}

\begin{example}\label{exa2.7} 
 Let $E\subset\mathbb T$ be a set of absolute convergence, 
or, in other terminology,   of type N (see \cite{8}, \cite{11}),
that is, there exists a sequence $\{a_n\}_{n=1}^\infty$
such that $a_n\geq 0$ for every $n\geq 1$, $\sum_{n=1}^\infty a_n = \infty$,
and $\sum_{n=1}^\infty a_n|\operatorname{Im} \zeta^n|< \infty$ for every $\zeta\in E$.
Note that we do not suppose that $E$ is closed. Moreover, if  
a sequence $\{a_n\}_{n=1}^\infty$ as above is fixed and the set
$E=\{ \zeta\in\mathbb T : \ \sum_{n=1}^\infty a_n|\operatorname{Im} \zeta^n|< 
\infty\}$ is infinite,
then $\operatorname{clos}E=\mathbb T$ {\cite[VII.XI]{8}}.

Let  $\mu$ be a finite positive Borel
measure on $\mathbb T$ such that $\mu(\mathbb T) =\mu(E)$.
Then there exists a sequence $\{n_k\}_k$ such that  
$\zeta^{n_k}\to_k1$ a.e. with respect to $\mu$
 (see {\cite[\S VII.7]{8}}, {\cite[XIII.2.3]{11}}).
We sketch the proof briefly. Let $\{a_n\}_{n=1}^\infty$ be a sequence 
such that $a_n\geq 0$ for every $n\geq 1$, $\sum_{n=1}^\infty a_n = \infty$,
and $\sum_{n=1}^\infty a_n|\operatorname{Im} \zeta^n|< \infty$ for a.e.  $\zeta$
with respect to $\mu$. By the dominated convergence theorem,
$$\int_{\mathbb T}\Bigl(\sum_{n=1}^N a_n|\operatorname{Im} \zeta^n|\Bigr)\big/
\Bigl(\sum_{n=1}^N a_n\Bigr)d\mu(\zeta)\to_{N\to\infty}0,$$ therefore,
$\liminf_n\int_{\mathbb T}|\operatorname{Im} \zeta^n|d\mu(\zeta)=0$.
Thus, there exists a sequence $\{n_k\}_k$ such that  
$\sum_k\int_{\mathbb T}|\operatorname{Im} \zeta^{n_k}|d\mu(\zeta) < \infty$. 
Since
$$\sum_k\int_{\mathbb T}|\operatorname{Im} \zeta^{n_k}|d\mu(\zeta) =
\int_{\mathbb T}(\sum_k|\operatorname{Im} \zeta^{n_k}|)d\mu(\zeta),$$
we conclude that $\sum_k|\operatorname{Im} \zeta^{n_k}|<\infty$ for 
 a.e. $\zeta$ with respect to $\mu$, therefore, $\zeta^{n_k}\to_k1$ a.e. with respect to $\mu$.

Now, let $U$ be  a unitary operator, $\mu$ be its  scalar spectral measure, and let
there exist a set $E$ of absolute 
convergence such that   
 $\mu(\mathbb T) =\mu(E)$.
Then, as it was proved just above,
 there exists a sequence $\{n_k\}_k$ such that  
$\zeta^{n_k}\to_k1$ a.e. with respect to $\mu$, therefore,
 $U^{n_k}\to_k I$ in the weak operator topology.
\end{example}

\begin{lemma}\label{lem2.8} Let $\mu$ be a finite positive Borel  measure on
 $\mathbb T$ such that
$$\limsup_n|\hat\mu(n)|=\mu(\mathbb T),$$
 where $\hat\mu(n)$ are the Fourier coefficients of $\mu$.
Then there exist    a  sequence $\{n_k\}_k$ and a point  $\xi\in\mathbb T$ 
such that $\zeta^{n_k}\to_k\xi$ in the weak-star topology on  $L^\infty(\mu)$.
\end{lemma}

\begin{proof} Without loss of generality, we may assume that $\mu(\mathbb T)=1$.
Let $\{\ell_k\}_k$ be a sequence such that 
$\limsup_n|\hat\mu(n)|=\lim_k|\hat\mu(\ell_k)|$.
There exist a subsequence  $\{n_k\}_k$ of $\{\ell_k\}_k$  and 
 a function $\varphi\in L^\infty(\mu)$
such that $\zeta^{n_k}\to_k\varphi$ 
in the  weak-star topology on  $L^\infty(\mu)$. Since
$\|\zeta^{n_k}\|_\infty=1$, we have  $\|\varphi\|_\infty\leq 1$.
Furthermore,
$$1=\lim_k|\hat\mu(n_k)|=\lim_k\Bigl|\int\zeta^{n_k}d\mu(\zeta)\Bigr|=
\Bigl|\int\varphi d\mu\Bigr|
\leq\int|\varphi| d\mu\leq 1.$$
Thus, $\int|\varphi| d\mu= 1$, therefore,  $|\varphi(\zeta)|=1$ 
for  a.e. $\zeta$ with respect to $\mu$. We set $\xi=\int\varphi d\mu$,
then 
$1=\int\overline\xi\varphi d\mu=
\int\operatorname{Re}(\overline\xi\varphi) d\mu$.
Since $\operatorname{Re}(\overline\xi\varphi)\leq 1$, we conclude that 
 $\varphi(\zeta)=\xi$ 
for  a.e. $\zeta$ with respect to $\mu$.  \end{proof}

\begin{corollary}\label{cor2.9} If the scalar spectral measure of a unitary operator $U$ 
satisfies the conditions of Lemma \ref{lem2.8}, then there exist a sequence
 $\{n_k\}_k$ and  a point $\xi\in\mathbb T$ 
 such that $U^{n_k}\to_k \xi I$ in the weak operator topology. \end{corollary}

\begin{example}\label{exa2.10}{\cite[Theorem 2]{4}}. There exist  unitary operators  $U$ 
such that    
$U^{n_k}\to_k\kappa I$ in the weak operator topology for some sequences  
$\{n_k\}_k$ 
and some $\kappa\in\mathbb C$, $0<|\kappa|<1$.
\end{example}

The following theorem is the main result of the paper.

\begin{theorem}\label{th2.11} Suppose that $M$, $C$ are positive constants,
$\{\mathcal K_j\}_j$ is  no more than countable family of Hilbert spaces,
and $U_j:\mathcal K_j\to\mathcal K_j$ are unitary operators. 
We set $\mathcal K=\oplus_j\mathcal K_j$ and
$U=\oplus_jU_j$. Suppose that

\noindent $\text{\rm(i)}$  every operator $U_j$ has the following property:
if $R$ is an operator such that $\sup_{n\geq 0}\|R^n\|\leq M$ and
$R\sim U_j$, then $\sup_{n\leq 0}\|R^n\|\leq C$.

\noindent $\text{\rm(ii)}$   $\mathcal K_j\in\operatorname{Hyplat}U$ for every $j$.

 Let $T$ be an operator such that  $\sup_{n\geq 0}\|T^n\|\leq M$
and $T\sim U$. Then $T\approx U$ and  $\sup_{n\leq 0}\|T^n\|\leq M^2C^3$.
\end{theorem}

\begin{proof} Since  $T\sim U$, we have that $U$ is unitarily equivalent to 
the isometric asymptote of $T$.
Denote by $\mathcal H$ the space on which $T$ acts, and by
 $X:\mathcal H\to\mathcal K$ the canonical intertwining mapping 
for $T$ and $U$ (see \cite{9}).
Denote by 
\begin{equation}\label{2.7} q: \operatorname{Hyplat}_1T\to\operatorname{Hyplat}U, 
\ \ \ q\mathcal M=\operatorname{clos}X\mathcal M,
\ \ \ \mathcal M\in\operatorname{Hyplat}_1T, \end{equation}
 the lattice isomorphism between  $\operatorname{Hyplat}_1T$ and 
 $\operatorname{Hyplat}U$
(see \cite{9}). We set $\mathcal M_j=q^{-1}\mathcal K_j$ and $T_j=T|_{\mathcal M_j}$ for every $j$.
By  \cite{9}, $T_j\sim U_j$, therefore, by assumption (i),
$T_j\approx U_j$ and  
\begin{equation}\label{2.8} \sup_{n\leq 0}\|T_j^n\|\leq C \ \ \text{ for every} \  j.   \end{equation}
It is obvious from the above relations that  $\oplus_jT_j\approx U$.
We will show that $T\approx\oplus_jT_j$.   

From the estimate on $\|T^n\|$ for $n\geq 0$, we have that $\|X\|\leq M$.
From assumption (i)  and the properties of $X$ (see \eqref{2.3} and \eqref{2.8}) we have that
$\|Xx\|\geq\frac{1}{C}\|x\|$ for every $x\in\mathcal M_j$ and every $j$.
Let $\{x_j\}_j$ be a {\it finite} family such that $x_j\in\mathcal M_j$. Then
$$\frac{1}{C^2}\sum_j\|x_j\|^2\leq\sum_j\|Xx_j\|^2=
\Bigl\|X\bigl(\sum_jx_j\bigr)\Bigr\|^2
\leq M^2\Bigl\|\sum_jx_j\Bigr\|^2.$$
We obtain that 
\begin{equation}\label{2.9} \begin{gathered}\frac{1}{M^2C^2}\sum_j\|x_j\|^2\leq\Bigl\|\sum_jx_j\Bigl\|^2  \\
\ \ \text{for any  finite family} \ \ 
\{x_j\}_j  \ \ \text{such that} \ \ x_j\in\mathcal M_j. \end{gathered}\end{equation}

Clearly, $U^\ast$ and $T^\ast$ satisfy the conditions of the theorem.
Denote by  $X_\ast:\mathcal H\to\mathcal K$ the canonical intertwining mapping 
for $T^\ast$ and $U^\ast$ and by $q_\ast$  a lattice isomorphism
 between  $\operatorname{Hyplat}_1T^\ast$ and
 $\operatorname{Hyplat}U^\ast=\operatorname{Hyplat}U$ 
 defined analogously to \eqref{2.7}.
We set $\mathcal M'_j=q_\ast^{-1}\mathcal K_j$  for every $j$.
Analogously to \eqref{2.9} we obtain that 
\begin{equation}\label{2.10} \begin{gathered} \frac{1}{M^2C^2}\sum_j\|y_j\|^2\leq\Bigl\|\sum_jy_j\Bigr\|^2 \\
\ \ \text{for any  finite family} \ \ 
\{y_j\}_j  \ \ \text{such that} \ \ y_j\in\mathcal M'_j.  \end{gathered}\end{equation}

Now, we will  show that the families $\{\mathcal M_j\}_j$ and $\{\mathcal M'_j\}_j$ are 
biorthogonal, that is, 
\begin{equation}\label{2.11} \mathcal M_j\perp\mathcal M'_k, \ \ \text{if} \ \ j\neq k. \end{equation}
Let $j\neq k$, and  let $x\in\mathcal M_j$ and $y\in\mathcal M'_k$.
By  \cite{9}, $q^{-1}\mathcal E= \operatorname{clos}X_\ast^\ast\mathcal E$
for every $\mathcal E\in\operatorname{Hyplat}U$.
Thus, $x=\lim_\ell X_\ast^\ast g_\ell$, where $g_\ell\in\mathcal K_j$, and
$(x,y)=\lim_\ell (X_\ast^\ast g_\ell,y)=\lim_\ell ( g_\ell,X_\ast y)=0$,
 because $X_\ast y\in \mathcal K_k$.

From  \eqref{2.10}, \eqref{2.11} and {\cite[VI.4]{15}},  {\cite[C.3.1]{16}} we conclude that 
\begin{equation}\label{2.12}  \begin{gathered}\Bigl\|\sum_jx_j\Bigr\|^2\leq M^2C^2\sum_j\|x_j\|^2 \\ \ \ 
\text{for any finite family} \ \ 
\{x_j\}_j  \ \ \text{such that} \ \ x_j\in\mathcal M_j.  \end{gathered}\end{equation}
For convenience, we give a  proof of \eqref{2.12} below.
By \eqref{2.10}, the mapping
 $$Z:\mathcal H\to \oplus_j\mathcal M'_j, \ \ \ Z\bigl(\sum_jy_j\bigr)=\oplus_jy_j, 
 \ \ \text{where} \ \
y_j\in\mathcal M'_j,$$
is  a linear bounded transformation and $\|Z\|\leq MC$.
Denote by $P_j:\mathcal M'_j\to\mathcal M_j$ the restriction on $\mathcal M'_j$ 
of the orthogonal projection
on $\mathcal M_j$. We set $P=\oplus_jP_j$, then, evidently, $\|P\|\leq 1$.
Let us regard the linear bounded transformation
$$(PZ)^\ast: \oplus_j\mathcal M_j\to\mathcal H.$$
Clearly, 
\begin{equation}\label{2.13} \|(PZ)^\ast\|=\|PZ\|\leq MC. \end{equation}
Using the biorthogonality of the families $\{\mathcal M_j\}_j$ and 
$\{\mathcal M'_j\}_j$,
it is easy to see that $(PZ)^\ast$ acts by the formula
$$(PZ)^\ast(\oplus_jx_j)=\sum_jx_j, \ \ \text{where} \ \
x_j\in\mathcal M_j.$$
Now, \eqref{2.12} follows from \eqref{2.13}.

From \eqref{2.9} and \eqref{2.12} we conclude that the mapping
$$Y:\mathcal H\to \oplus_j\mathcal M_j, \ \ \ Y\bigl(\sum_jx_j\bigr)=\oplus_jx_j, 
\ \ \text{where} \ \
x_j\in\mathcal M_j,$$
is a  linear bounded invertible transformation, $\|Y\|\leq MC$, 
$\|Y^{-1}\|\leq MC$, and, evidently,
$YT=(\oplus_jT_j)Y$.
Therefore,   $T\approx\oplus_jT_j$ and 
$$\|T^n\|\leq\|Y^{-1}\|\, \|(\oplus_jT_j)^n\| \|Y\|\leq M^2C^3 \ \text{ for all } 
\ n\leq 0.$$ \end{proof}

The following theorem can be regarded as a generalization of Lemma \ref{lem1.1}.
It is evident from the definition of the Wiener algebra $A(\mathbb T)$ of
functions on $\mathbb T$ that $A(\mathbb T)$ is isometrically isomorphic 
to the space $\ell^1$;
therefore, the dual space $A(\mathbb T)^\ast$ of $A(\mathbb T)$ is 
  isometrically isomorphic to $\ell^\infty$. The sequences from
$\ell^\infty$ regarded as elements of $A(\mathbb T)^\ast$ are called
 {\it pseudomeasures} \cite{2}, \cite{6}, \cite{7}, \cite{8}, \cite{11}.
The inclusion $A(\mathbb T)\subset C(\mathbb T)$
implies that $C(\mathbb T)^\ast\subset A(\mathbb T)^\ast$, that is,
if $\mu$ is a {\it complex} Borel measure on $\mathbb T$, then 
$\mu\in A(\mathbb T)^\ast$, and for $f=\sum_{n\in\mathbb Z}\hat f(n)\zeta^n\in A(\mathbb T)$
we have 
$$\langle f,\mu\rangle=\int_{\mathbb T}f(\zeta)d \mu(\zeta)=
\int_{\mathbb T}\big(\sum_{n\in\mathbb Z}\hat f(n)\zeta^n\big)d \mu(\zeta)
=\sum_{n\in\mathbb Z}\hat f(n)\hat \mu(-n),$$
where $\hat\mu(n)=\int_{\mathbb T}\zeta^{-n}d\mu(\zeta)$. In other words,
a  complex Borel measure on $\mathbb T$
is a pseudomeasure. But there exist pseudomeasures that are not 
 measures, that is, there exist sequences from $\ell^\infty$ that are not 
sequences of the  Fourier coefficients of any complex Borel measure  on $\mathbb T$
 \cite{2}, \cite{6},  \cite{8}, \cite{11}.

The duality between functions $f=\sum_{n\in\mathbb Z}\hat f(n)\zeta^n\in A(\mathbb T)$
and pseudomeasures $\mu=\{\hat \mu(n)\}_{n\in\mathbb Z}\in\ell^\infty$ is
given by the formula $\langle f,\mu\rangle=\sum_{n\in\mathbb Z}\hat f(n)\hat \mu(-n)$.
The following description of  $AA^+$-sets is a consequence of this duality.
A closed set $E\subset\mathbb T$ is an $AA^+$-set if and only if
there exists a constant $K(E)$ such that 
\begin{equation}\label{2.14}  \sup_{n\in\mathbb Z}|\hat\mu(n)|\leq K(E)\limsup_{n\to\infty}|\hat\mu(n)| 
 \end{equation}
for every  pseudomeasure $\mu$ such that 
$\langle f,\mu\rangle=0$ for every $f\in I(E)$ (see  \cite{6}, \cite{7}). 
Theorem \ref{th2.12} shows that in Lemma \ref{lem1.1}  condition \eqref{2.14} on
 {\it pseudomeasures}
(which is contained in Lemma \ref{lem1.1} in a  nonobvious form) can be replaced 
by the condition on {\it  positive measures}. Note that for  positive measures
 on some set $E$  $\sup_{n\in\mathbb Z}|\hat\mu(n)|=\mu(E)$.

\begin{theorem}\label{th2.12} Let $K>0$ be a constant, and let  $E\subset\mathbb T$ 
be a Borel set such that $\mu(E)\leq K\limsup_{n\to \infty}|\hat\mu(n)|$ 
for every  positive finite Borel measure $\mu$ on 
$\mathbb T$ such that
$\mu(E)=\mu(\mathbb T)$. 
Let $U$ be a unitary operator, let $\nu$ be the scalar spectral measure of $U$,
and let $\nu(E)=\nu(\mathbb T)$.
Suppose that $T$ is a power bounded operator such that  $T\sim U$.
Then $T\approx U$ and
  $$\sup_{n\leq 0}\|T^n\|\leq K^3(\sup_{n\geq 0}\|T^n\|)^8.$$
\end{theorem}

 \begin{proof}  Let $E$ be a set having the property described in the theorem, 
 let $\nu$ be a positive finite Borel measure on $\mathbb T$  such that 
$\nu(E)=\nu(\mathbb T)$ , and let $0<c<\frac{1}{K}$.
Applying the Zorn lemma, it is easy to see that 
there exists no more than countable family $\{E_j\}_j$ of 
Borel sets  such that $E=\cup_jE_j$ and for every $j$ there exist 
a sequence $\{n_{jk}\}_k$ and a function
$\varphi_j\in L^\infty(E_j, \nu)$ 
such that $\zeta^{n_{jk}}\to_k \varphi_j$ in the weak-star
topology on $ L^\infty(E_j, \nu)$ and 
$|\varphi_j|\geq c$ a.e. on $E_j$ with respect to $\nu$.

Now let $\nu$ be the scalar spectral measure of a unitary operator $U$,
and let $0<c<\frac{1}{K}$ be fixed. Let $\{E_j\}_j$ be the family of Borel sets
constructed just above.
We set  $E'_1=E_1$ and 
$E' _j=E _j\setminus E _{j-1}$ for $j>1$; of course,  the sets $\{E'_j\}_j$ 
are mutually disjoint.
For every $j$, denote by $U_j$ the restriction of $U$ on its spectral 
subspace corresponding to the set $E'_j$.
Clearly, $U_j$ and $\varphi_j(U_j)$ satisfy the conditions of Theorem \ref{th2.5}, 
and $\|(\varphi_j(U_j))^{-1}\|\leq\frac{1}{c}$ for every $j$. Therefore, 
$U=\oplus_j U_j$ satisfies the conditions of Theorem \ref{th2.11}.
We set $M=\sup_{n\geq 0}\|T^n\|$; by Theorems  \ref{th2.5} and \ref{th2.11}, 
 $T\approx U$ and  
$$\sup_{n\leq 0}\|T^n\|\leq M^2(\frac{1}{c}M^2)^3
= \frac{1}{c^3}M^8.$$
To conclude the proof of the theorem, let  $\frac{1}{c}$ tend to $K$.  \end{proof}

\section{On the sets satisfying some metric condition}

In this section we apply the main result of the paper to unitary operators 
whose spectral measures are supported on sets satisfying some 
metric condition. 

\begin{definition}\label{def3.1} \cite{5}, \cite{7}.   Let $E\subset\mathbb T$ be a closed set. 
For $\varepsilon>0$,
 denote by $N_\varepsilon$ the smallest number of closed arcs 
of length $\varepsilon$ whose union contains $E$. We set
$$\alpha(E)=\liminf_{\varepsilon\to 0}
\frac{N_\varepsilon}{\log\frac{1}{\varepsilon}}.$$
\end{definition}

It is proved in  \cite{5}, \cite{7}  that if $\alpha(E)<\infty$, 
then $E$ is an $AA^+$-set; 
however, the norm $K(E)$ of the inverse to the imbedding  \eqref{1.1}  
can be arbitrarily large.
In this section we  show that for $AA^+$-sets $E$ such that
  $\alpha(E)<\infty$
the estimate of the left part of  \eqref{1.2}  does not depend 
on $K(E)$.

We will need the following technical lemma.

\begin{lemma}\label{lem3.2} Let $E\subset\mathbb T$ be a closed set such that 
$\alpha(E)<\infty$, and let $\delta>0$. Then
$$E=\bigcup_{j=1}^\ell\{\xi_j\}\cup\bigcup_kE_k,$$
where $\ell<\infty$, $\xi_j\in E$, $j=1,\ldots, \ell$, and
$\{E_k\}_k$ is no more than countable family of closed mutually disjoint
subsets of $E$ such that $\alpha(E_k)\leq\delta$.
\end{lemma}

\begin{proof} First, we prove the lemma for $\delta=\alpha(E)/2$.
If $E'\subset\mathbb T$ is  a closed set which 
does not contain nonempty open arcs and $d>0$, then 
$E'$ can be represented in the form
$E'=\cup_{k=1}^KE'_k$, where $K<\infty$ and 
$\{E'_k\}_{k=1}^K$ is a family of closed mutually disjoint
subsets of $E'$ such that $\operatorname{diam} E'_k\leq d$, $k=1,\ldots,K$
(recall that 
$\operatorname{diam} E=\sup\{|\zeta-\xi| :\ \  \zeta,\ \xi\in E\}$).

Let a sequence $\{d_n\}_n$ be such that $d_n>0$, $n=1,2,\ldots$,
and $d_n\to_n 0$.
We represent $E$ in the form  $E=\cup_{k=1}^{K_1}E_{1k}$,
 where $\{E_{1k}\}_{k=1}^{K_1}$ is a family 
of closed mutually disjoint nonempty subsets of $E$ such that 
$\operatorname{diam} E_{1k}\leq d_1$, $k=1,\ldots,K_1$.
If $\varepsilon>0$ is less than the minimum of distances
 between the sets $E_{1k}$, then
 $N_\varepsilon(E)=\sum_{k=1}^{K_1} N_\varepsilon(E_{1k})$.
Therefore, 
$\alpha(E)\geq\sum_{k=1}^{K_1} \alpha(E_{1k})$.
From this estimate we conclude that
 there is at most one index $k$ such that
$\alpha(E_{1k})>\alpha(E)/2$. 
If  $\alpha(E_{1k})\leq\alpha(E)/2$ for every  index $k$, then the lemma 
is proved for  $\delta=\alpha(E)/2$. 
If there exists an index  $k$,  $1\leq k\leq K_1$, such that
$\alpha(E_{1k})>\alpha(E)/2$, say $k=1$, then  
we apply the above procedure to the number $d_2>0$ and the set $E_{11}$,
and so on. The following two cases are possible.

{\it The first case} : in some step, say $n$,  we have that 
$\alpha(E_{nk})\leq\alpha(E)/2$ for every  index $k$, $k=1,\ldots,K_n$,
thus, the lemma is proved for  $\delta=\alpha(E)/2$. 

{\it The second case} : there exists a sequence of closed sets $\{E_{n1}\}_n$
 such that $\operatorname{diam} E_{n1}\leq d_n$, $E_{n+1, 1}\subset E_{n1}$, 
the set
$E'_n=E_{n1}\setminus E_{n+1, 1}$ is a closed nonempty set, and 
$\alpha(E_{n1})>\alpha(E)/2$ for every  $n$. Since $d_n\to_n 0$, 
the intersection
$ \cap_n  E_{n1}$ is a singleton. Finally, 
$E_{n+1, 1}\cup E'_n=E_{n1}\subset E$, 
and, since $E_{n+1, 1}$ and $ E'_n$ are closed and disjoint, 
$$\alpha(E)\geq\alpha(E_{n1})\geq\alpha(E_{n+1, 1})+\alpha( E'_n)>
\alpha(E)/2+\alpha( E'_n),$$
 and we conclude that $\alpha( E'_n)\leq\alpha(E)/2$.

Thus, the lemma is proved for $\delta=\alpha(E)/2$, that is, 
the representation of $E$ 
in the form
$$E=\{\xi\}\cup\bigcup_kE_k,$$
where $\{E_k\}_k$ is no more than countable family of closed mutually disjoint
subsets of $E$ such that $\alpha(E_k)\leq\alpha(E)/2$ and 
$\{\xi\}= \cap_n  E_{n1}$,
if the second case takes place, is obtained. Now we apply 
the already proved part of
the lemma to every set $E_k$ and obtain the needed representation of $E$ 
in which (new) sets $E_k$ satisfy the condition  $\alpha(E_k)\leq\alpha(E)/4$,
and so on. On some step, say $j$, we get that 
$\alpha(E)/2^j\leq\delta$.  \end{proof} 

\begin{remark}\label{rem3.3} In the proof of Lemma \ref{lem3.2}, the second case actually can take place. 
Let $0<a\leq 1/2$ and 
$E_j=\{e^{i\frac{a}{1-a}}\}\cup\bigcup_{n=j}^\infty\{e^{i\sum_{k=1}^n a^k}\}$.
It is easy to see that $\alpha(E_j)=\frac{1}{\log\frac{1}{a}}$ for every 
$j=1,2,\ldots$
and $\cap_{j=1}^\infty E_j=\{e^{i\frac{a}{1-a}}\}$.
\end{remark}

The following lemma is actually from \cite{7} and {\cite[\S VII.8]{8}}. 
For convenience, we sketch the proof.

\begin{lemma}\label{lem3.4} Suppose that  $K\in\mathbb N$, $K\geq 3$, and 
$E\subset\mathbb T$ is a closed set 
such that $\alpha(E)<1/\log K$. 
Then $\liminf_n\|\zeta^n-1\|_{C(E)}\leq 2\sin\frac{\pi}{K-1}$.
\end{lemma}

\begin{proof}  There exists a sequence $\{\varepsilon_n\}_n$ 
such that $\varepsilon_n>0$,
$\varepsilon_n\to_n 0$ and 
\begin{equation}\label{3.1} N_n=N_{\varepsilon_n}\leq\frac{1}{\log K}\log\frac{1}{\varepsilon_n}. \end{equation}
We take a sequence $\{Q_n\}_n\subset \mathbb N$ such that $Q_n\to_n\infty$ and
\begin{equation}\label{3.2} Q_n\Bigl(\frac{K-1}{K}\Bigr)^{N_n}\to_n 0. \end{equation}
Let $n$ be fixed. Denote by $t_{n1}, \ldots, t_{nN_n}$  real points such that 
 $e^{2\pi i t_{nj}}$, $j=1, \ldots, N_n$,
 are the centerpoints of closed arcs of length $\varepsilon_n$     
 whose union contains $E$. We apply the Dirichlet theorem 
(see {\cite[appendix \S V.2]{8}}) to real points $t_{n1}, \ldots, t_{nN_n}$
and natural numbers $K-1$ and $Q_n$. We obtain $q_n\in\mathbb N$ and
 $p_{n1}, \ldots, p_{nN_n}\in\mathbb Z$ such that 
\begin{equation}\label{3.3} Q_n\leq q_n\leq Q_n (K-1)^{N_n}\ \ \text{and}
\ \ \ |q_nt_{nj}-p_{nj} |\leq \frac{1}{K-1},  \ \ j=1, \ldots, N_n. \end{equation}
Let $\zeta\in E$. There exist a real point $t$ and an index $j$,
 $1\leq  j\leq N_n$, such that
$\zeta=e^{2\pi i t}$ and $|t-t_{nj}|\leq\varepsilon_n/2$.
From this  estimate, \eqref{3.1} and \eqref{3.3}, we have  
\begin{equation}\label{3.4} | q_nt-p_{nj}|\leq 
\frac{Q_n}{2}\Bigl(\frac{K-1}{K}\Bigr)^{N_n}+\frac{1}{K-1}.\end{equation}
Since $\zeta=e^{2\pi i t}$ is an arbitrary point of $E$, from \eqref{3.2} and \eqref{3.4} 
we conclude that $\lim_n\|\zeta^{q_n}-1\|_{C(E)}\leq 2\sin\frac{\pi}{K-1}$. \end{proof}

\begin{theorem}\label{th3.5} Suppose that $\{E_k\}_k$ is no more than countable 
family of 
closed subsets of $\mathbb T$ such that $\alpha(E_k)<\infty$ 
(where the quantity 
$\alpha$ is defined in Definition \ref{def3.1}). 
 Suppose that $U$ is a unitary operator, $\mu$ is its scalar  spectral measure, 
 and $\mu(\mathbb T)=\mu(\cup_k E_k)$. 
Let $T$ be a power bounded operator such that $T\sim U$.   
 Then $T\approx U$ and 
 $\sup_{n\leq 0}\|T^nx\|\leq (\sup_{n\geq 0}\|T^nx\|)^8$.
\end{theorem}

\begin{proof}
 We set $M=\sup_{n\geq 0}\|T^nx\|$ and $E=\cup_k E_k$.
We denote by $\mu_a$ and $\mu_c$ the pure atomic and continuous parts of $\mu$, 
respectively, and we denote by $U_a$ and $U_c$ the correspondent parts of $U$.
It is easy to see that $U_a$ satisfies  condition (i) of Theorem \ref{th2.11}
with $C=M^2$ (see Example \ref{exa2.6}).

We fix a natural number $K\geq 3$. Applying Lemma \ref{lem3.2} 
with $\delta<1/\log K$ to every 
set $E_k$,  we obtain the representation of $E$
in the form
\begin{equation}\label{3.5} E=\bigcup_j\{\xi_j\}\cup\bigcup_\ell E_{K\ell}, \end{equation}
where $\alpha( E_{K\ell})<1/\log K$ and the sets of indeces $j$ and $\ell$ are 
no more than countable. Let $U_{K\ell}$ be the restriction 
of $U_c$ on its spectral subspace corresponding to the set $E_{K\ell}$.
Applying Theorem \ref{th2.5} and Lemma \ref{lem3.4}, we will show that  
$U_{K\ell}$ satisfies condition (i) of Theorem \ref{th2.11}
with
\begin{equation}\label{3.6} C=C_K=\frac{1}{1-2\sin\frac{\pi}{K-1}}M^2. \end{equation}
By  Lemma \ref{lem3.4},  there exists a sequence $\{q_n\}_n$
(which depends on $K$ and $\ell$) such that
$\lim_n\|\zeta^{q_n}-1\|_{C(E_{K\ell})}\leq 2\sin\frac{\pi}{K-1}$.
There exist a subsequence  $\{p_n\}_n$ of $\{q_n\}_n$  and 
 a function $\varphi\in L^\infty(E_{K\ell}, \mu_c)$
such that $\zeta^{p_n}\to_n\varphi$ 
in the weak-star topology on  $L^\infty(E_{K\ell}, \mu_c)$. 
Thus,
$U_{K\ell}^{p_n}\to_n \varphi(U_{K\ell})$ in the weak operator topology.
Since
$\lim_n\|\zeta^{p_n}-1\|_\infty\leq 2\sin\frac{\pi}{K-1}$, we have that 
$|\varphi|\geq 1-2\sin\frac{\pi}{K-1}$ a.e. on $E_{K\ell}$ 
with respect to $\mu_c$.  
By Theorem \ref{th2.5}, 
$U_{K\ell}$ satisfies condition (i) of Theorem \ref{th2.11}
with $C=C_K$.

Since we do not suppose that the sets $E_k$ are mutually disjoint, 
we need to change the sets $E_{K\ell}$ in representation \eqref{3.5}.
Namely, we set $E'_{K1}=E_{K1}$ and 
$E' _{K\ell}=E _{K\ell}\setminus E _{K,\ell-1}$ for $\ell>1$.
Now the sets $E' _{K\ell}$ are mutually disjoint.
Let $U'_{K\ell}$ be the restriction 
of $U_c$ on its spectral subspace corresponding to the set $E'_{K\ell}$.
Applying  Lemma \ref{lem2.2}  to  $U_{K\ell}$ we conclude that  
$U'_{K\ell}$ satisfies the condition (i) of Theorem \ref{th2.11}
with the same constant $C=C_K$.
Thus,  $U$ has  the representation
$$U=U_a\oplus\bigoplus_\ell U'_{K\ell},$$ 
and we conclude that $U$
 satisfies the conditions of Theorem \ref{th2.11}. Therefore, $T\approx U$ and 
 $$\sup_{n\leq 0}\|T^nx\|\leq M^2C_K^3, $$
 where $C_K$ is defined in \eqref{3.6} and $K$ is an arbitrary natural number, $K\geq 3$.

To conclude the proof of the theorem, we note that
$ M^2C_K^3\to M^8$ when $K\to \infty$. \end{proof}

\begin{corollary}\label{cor3.6} In Theorem \ref{th3.5} the condition $T\sim U$ 
can be replaced  by $T\prec U$.\end{corollary}

\begin{proof}  We will show that $U$ satisfies  condition 1) of Lemma \ref{lem2.4}.
Then the corollary will be proved.  
Denote by $\mathcal K$ the space on which $U$ acts and by $\mu$ the scalar spectral measure of $U$.
Following \cite{12}
we set $$\mathcal Z(U)=\{x\in\mathcal K: U^nx\to_n 0\ \text{  in the weak topology}\}.$$
It is easy to see that $\mathcal Z(U)$ is a hyperinvariant subspace for $U$, 
therefore, there exists a set $\tau\subset\mathbb T$ such that
$\mathcal Z(U)$
 is the spectral subspace of $U$ corresponding to $\tau$.
Let us suppose that $\mu(\tau)>0$.
Then there exists an index $k$ such that $\mu(\tau\cap E_k)>0$.
The  spectral subspace of $U$ corresponding to
 $\tau\cap E_k$ is contained in $\mathcal Z(U)$, therefore, the
Fourier coefficients of $\mu|_{\tau\cap E_k}$ tend to zero. But, 
by assumption, the set $E_k$ is an $AA^+$-set, and if the Fourier 
coefficients of a measure supported on  $E_k$ tend to zero, 
then this measure must be zero itself \cite{6}, \cite{7} (see \eqref{2.14}), a contradiction. 
\end{proof}


\begin{thebibliography}{HD}




\normalsize
\baselineskip=11pt


\bibitem[1]{1} T. Ando and K. Takahashi, On operators with unitary $\rho$-dilations,
{\it Ann. Polon. Math.}, {\bf 66} (1997), 11--14.

\

\bibitem[2]{2}  C. C. Graham and O. C. McGehee, {\it Essays in commutative harmonic analysis}, 
Springer, Berlin, Heidelberg, New York, 1979.

\

\bibitem[3]{3}  K. Hoffman, {\it Banach spaces of analytic functions}, Prentice-Hall, Englewood Cliffs, N.J., 1962.

\

\bibitem[4]{4} A. Iwanik, M. Lema\'nczyk and C. Mauduit, Piecewise absolutely continuous cocycles over
 irrational rotations, {\it J. London Math. Soc. (2)}, {\bf 59} (1999), 171--187. 

\

\bibitem[5]{5} J. P. Kahane, A metric condition for a closed circular set to be a set of uniqueness, 
{\it J. Approx. Theory}, {\bf 2} (1969), 233--236.

\

\bibitem[6]{6} J. P. Kahane, {\it S\'eries de Fourier absolument convergentes},
Springer-Verlag, Berlin, Heidelberg, New York, 1970. 

\

\bibitem[7]{7}  J. P. Kahane and Y. Katznelson, Sur les alg\`ebres de restrictions des s\'eries de Taylor
absolument convergentes \`a un ferm\'e du cercle, {\it J. Analyse Math.},
 {\bf 23} (1970), 185--197.

\

\bibitem[8]{8}  J. P. Kahane and R. Salem, {\it Ensembles parfaits et s\'eries trigonom\'etriques},
Hermann, Paris, 1994.
  
\

\bibitem[9]{9}  L. K\'erchy, Isometric asymptotes of power bounded operators, 
    {\it Indiana Univ. Math. J.}, {\bf 38} (1989), 173--188. 

\

\bibitem[10]{10}  L. K\'erchy, On the inclination of hyperinvariant subspaces 
   of $C_{11}$-contractions, {\it The Gohberg Anniversary Collection, 
   Volume II: Topics in Analysis and Operator Theory},  
   Oper. Theory Adv. Appl., Birkh\"auser, Basel, {\bf 41} (1989), 345--351. 

\

\bibitem[11]{11} L. A. Lindahl and F. Poulsen (Eds.), {\it Thin sets in harmonic analysis}, Marcel Decker,
New York, 1971.

\

\bibitem[12]{12} W. Mlak, Decompositions of polynomially bounded operators, 
{\it Bull. Acad. Polon. Sci., S\'er. Sci. Math. Astronom. Phys.}, {\bf 21}
  (1973), 317--322.

\

\bibitem[13]{13} W. Mlak, Algebraic polynomially bounded operators, 
{\it Ann. Polon. Math.}, {\bf 29} (1974), 133--139.
          	
\
   
\bibitem[14]{14} V. M\"uller and Y. Tomilov, Quasisimilarity of power bounded operators
and Blum--Hanson property, {\it J. Funct. Anal.}, {\bf 246} (2007), 385--399.  	   
 
\

\bibitem[15]{15}N. K. Nikol'skii, {\it Treatise on the shift operator}, Springer, 
Berlin, 1986.  

\

\bibitem[16]{16} N. K. Nikolski, {\it  Operators, functions, and systems: an easy reading.
 Volume II:
Model operators and systems}, Math. Surveys and Monographs 
{\bf 93}, AMS, 2002    

\

 \bibitem[17]{17}  M. Radjabalipour, Some results on power bounded operators, 
    {\it Indiana Univ. Math. J.}, {\bf 22} (1973), 673--677.

\

\bibitem[18]{18} B. Sz.-Nagy, On uniformly bounded linear transformations in Hilbert space, {\it Acta Sci. Math. (Szeged)}, {\bf 11} (1947), 152--157.

\

\bibitem[19]{19} B. Sz.-Nagy and C. Foias, {\it Harmonic analysis of operators on 
Hilbert spaces}, North Holland, Amsterdam,  1970.

\end{thebibliography}
\end{document}